\documentclass[12pt,a4paper]{article}
\usepackage[utf8]{inputenc}
\usepackage[T1]{fontenc}
\usepackage{amsmath}
\usepackage{amsfonts}
\usepackage{amssymb}
\usepackage{graphicx}
\usepackage{amsthm}
\usepackage{tikz-cd}
\usepackage{comment}
\usepackage{mathrsfs}
\usepackage{hyperref}

\newcommand{\trop}{\operatorname{trop}}
\newcommand{\rk}{\operatorname{rk}}
\newcommand{\Gr}{\operatorname{Gr}}
\newcommand{\Fl}{\operatorname{Fl}}
\newcommand{\im}{\operatorname{im}}
\newcommand{\pos}{\operatorname{pos}}
\newcommand{\supp}{\operatorname{supp}}
\newcommand{\Mon}{\operatorname{Mon}}

\newtheorem{manualtheoreminner}{Theorem}
\newenvironment{manualtheorem}[1]{%
	\renewcommand\themanualtheoreminner{#1}%
	\manualtheoreminner
}{\endmanualtheoreminner}

\newtheorem{theorem}{Theorem}[section]
\newtheorem{proposition}[theorem]{Proposition}
\newtheorem{proposition definition}[theorem]{Proposition-Definition}
\newtheorem{corollary}[theorem]{Corollary}
\newtheorem{lemma}[theorem]{Lemma}

\theoremstyle{definition}

\newtheorem{definition}[theorem]{Definition}
\newtheorem{example}[theorem]{Example}
\newtheorem{remark}[theorem]{Remark}
\newtheorem{question}[theorem]{Question}
\newtheorem{problem}[theorem]{Problem}

\title{Realizability of matroid quotients}
\author{Alessio Borzì}

\begin{document}
	
	\maketitle
	
	{\raggedleft \footnotesize \textit{Vita quieta, mente lieta, moderata dieta.} \par}
	
	\begin{abstract}
		We characterize the realizability of a quotient of matroids, over an infinite field $K$, in terms of the realizability over $K$ of a single matroid associated to it, called the \emph{Higgs major}. This result extends to realizability of flag matroids. Further, we provide some applications to the relative realizability problem for Bergman fans in tropical geometry.
	\end{abstract}
	
	\section{Introduction}
	
	Matroids are central objects in modern combinatorics. A matroid $M$ can be seen as a combinatorial abstraction of a linear space. In a similar way, a \emph{matroid quotient} $M \twoheadrightarrow N$ can be seen as a combinatorial abstraction of an inclusion of two linear spaces. Even if $M$ and $N$ are realizable matroids, this does not guarantee that the quotient as a whole is realizable. An example of this was given in \cite[§1.7.5. Example 7]{borovik2003coxeter}, and is as follows. Let $P$ be the non-Pappus matroid, and let $e$ be one of the points of the "non-line". Then, the contraction $P / e$ is a matroid quotient of the deletion $P \setminus e$ and both matroids are realizable. However, the matroid quotient $P \setminus e \twoheadrightarrow P /e$ is not realizable. It is interesting to note that the non-realizability of the above quotient depends on the non-realizability of the non-Pappus matroid.
	
	
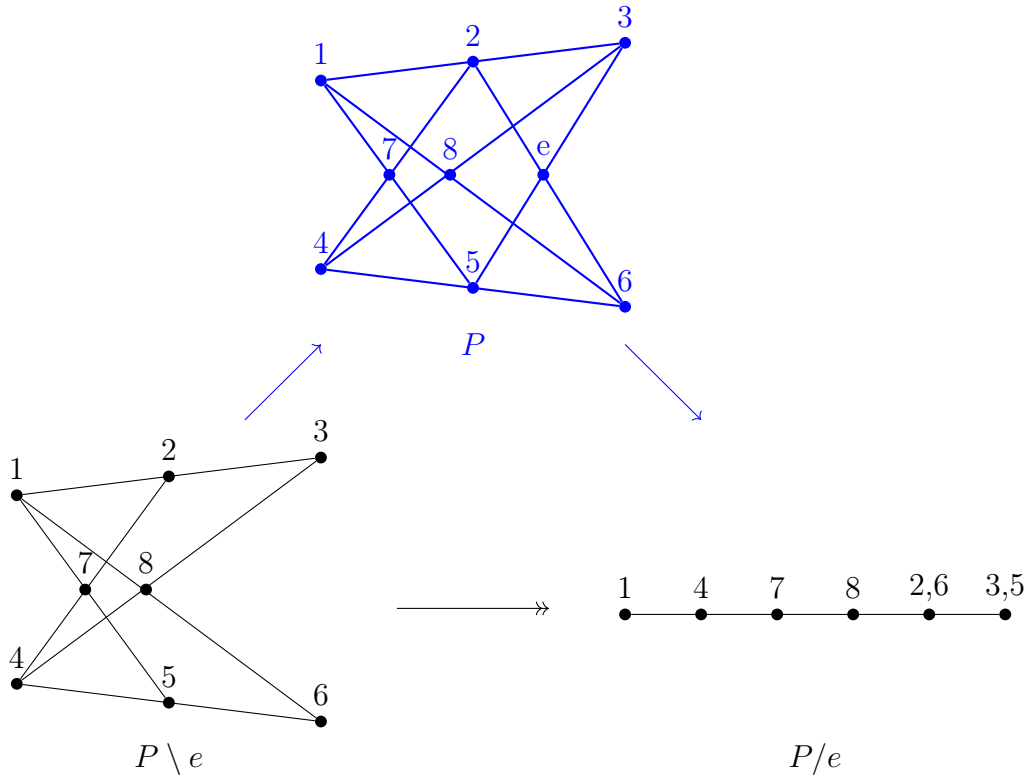
\begin{center}
			\begin{tikzpicture}
				\tikzstyle{circle}=[scale=0.4, shape=circle,fill=blue]
		
				
				\node (0) at (0,3) [circle] {};
				\node (1) at (2,3.25) [circle] {};
				\node (2) at (4,3.5) [circle] {};
				\node (3) at (0,0.5) [circle] {};
				\node (4) at (2,0.25) [circle] {};
				\node (5) at (4,0) [circle] {};
				\node (6) at (0.9,1.75) [circle] {};
				\node (7) at (1.7,1.75) [circle] {};
				\node (8) at (2.925,1.75) [circle] {};
				
				\draw[blue, thick] (0) -- (4);
				\draw[blue, thick] (2) -- (4);
				\draw[blue, thick] (1) -- (5);
				\draw[blue, thick] (1) -- (3);
				\draw[blue, thick] (0) -- (5);
				\draw[blue, thick] (2) -- (3);
				\draw[blue, thick] (0) -- (2);
				\draw[blue, thick] (3) -- (5);
				
				\node[blue] (A) at (0,3.35) {1};
				\node[blue] (B) at (2,3.6) {2};
				\node[blue] (C) at (4,3.85) {3};
				\node[blue] (D) at (0,0.85) {4};
				\node[blue] (E) at (2,0.6) {5};
				\node[blue] (F) at (4,0.35) {6};
				\node[blue] (G) at (0.9,2.1) {7};
				\node[blue] (H) at (1.7,2.1) {8};
				\node[blue] (I) at (2.925,2.1) {e};
				
				\node[blue] at (2, -0.5) {$P$};
				
				
				\tikzstyle{circle}=[scale=0.4, shape=circle,fill=black]
				
			 \def\x{4};
			 \def\y{5.5};
				
				\node (0) at (0-\x,3-\y) [circle] {};
				\node (1) at (2-\x,3.25-\y) [circle] {};
				\node (2) at (4-\x,3.5-\y) [circle] {};
				\node (3) at (0-\x,0.5-\y) [circle] {};
				\node (4) at (2-\x,0.25-\y) [circle] {};
				\node (5) at (4-\x,0-\y) [circle] {};
				\node (6) at (0.9-\x,1.75-\y) [circle] {};
				\node (7) at (1.7-\x,1.75-\y) [circle] {};
				
				\draw (0) -- (4);
				\draw (1) -- (3);
				\draw (0) -- (5);
				\draw (2) -- (3);
				\draw (0) -- (2);
				\draw (3) -- (5);
				
				\node (A) at (0-\x,3.35-\y) {1};
				\node (B) at (2-\x,3.6-\y) {2};
				\node (C) at (4-\x,3.85-\y) {3};
				\node (D) at (0-\x,0.85-\y) {4};
				\node (E) at (2-\x,0.6-\y) {5};
				\node (F) at (4-\x,0.35-\y) {6};
				\node (G) at (0.9-\x,2.1-\y) {7};
				\node (H) at (1.7-\x,2.1-\y) {8};
				
				\node at (2-\x, -0.5 - \y) {$P \setminus e$};
				
				
				\def\x{4};
				\def\y{-5.5};
				
				\node (0) at (0+\x,1.42+\y) [circle] {};
				\node (1) at (1+\x,1.42+\y) [circle] {};
				\node (2) at (2+\x,1.42+\y) [circle] {};
				\node (3) at (3+\x,1.42+\y) [circle] {};
				\node (4) at (4+\x,1.42+\y) [circle] {};
				\node (5) at (5+\x,1.42+\y) [circle] {};
				
				\draw (0) -- (5);
				
				\node (A) at (0+\x,1.77+\y) {1};
				\node (B) at (1+\x,1.77+\y) {4};
				\node (C) at (2+\x,1.77+\y) {7};
				\node (D) at (3+\x,1.77+\y) {8};
				\node (E) at (4+\x,1.77+\y) {2,6};
				\node (F) at (5+\x,1.77+\y) {3,5};
				
				\node at (2.5+\x, -0.5 + \y) {$P / e$};
				
				
				\draw[blue] [->] (-1,-1.5) -- (0,-0.5);
				\draw[blue] [->] (4,-0.5) -- (5,-1.5);
				\draw [->>] (1,1.5+\y) -- (3, 1.5+\y);
				
			\end{tikzpicture}
	\end{center}
	
	The aim of this paper is to characterize the realizability of a matroid quotient in terms of the realizability of a matroid associated to it. We will do so by using the notions of \emph{factorization} and \emph{major} of a matroid quotient, both appearing in a work of Kennedy \cite{kennedy1975majors}. The matroid we will associate to a matroid quotient is called the \emph{Higgs major} (see also \cite{Kung1986bookchapter}).
	
	\begin{manualtheorem}{A}[Theorem \ref{thm: characterization of realizability of quotients}]
		Let $f:M_1 \twoheadrightarrow M_2$ be a quotient of matroids, and let $K$ be an infinite field. The following statements are equivalent:
		\begin{enumerate}
			\item $f$ is realizable over $K$,
			\item the Higgs factorization of $f$ is realizable over $K$,
			\item the Higgs major of $f$ is realizable over $K$.
		\end{enumerate}
	\end{manualtheorem}
	
	A natural context in which to consider matroid quotients is flag varieties. A sequence of matroid quotients $M_1 \twoheadrightarrow M_2 \twoheadrightarrow \dots \twoheadrightarrow M_n$ is a  \emph{flag matroid}. Flag matroids are related to flag varieties in the same natural way as matroids are related to Grassmannians (see for instance \cite{gelfand1987combinatorial} and \cite{cameron2017flag}).
	
	Tropical analogues of flag varieties were studied in \cite{brandt2021tropical, borzi2023linear, bossinger2017computing, haque2012tropical}. In particular, in \cite{brandt2021tropical} a notion of \emph{flag Dressian} was introduced, and in order to compare it with the tropical flag variety, it was proved that for ground sets of cardinality up to five, all flag matroids are realizable. This result, in the non-valuated case, can be now seen as an immediate consequence of the above theorem (in the form of Corollary \ref{cor: realizability of flags}), combined with the fact that matroids with ground set of small enough cardinality are realizable. A similar reasoning might be applied accordingly for the linear degenerate flag Dressian defined in \cite{borzi2023linear}. 
	
	Another natural application of our main theorem is about relative realizability in tropical geometry (see for instance \cite[Question C]{jell2022moduli} and \cite[Question 1.1]{geiger2020realizability}). 
	
	\begin{problem}[Relative realizability problem]\label{problem: relative realizability}
		Given a pair of tropical varieties $\mathcal{Y} \subseteq \mathcal{X}$ and an algebraic variety $X$ tropicalizing to $\mathcal{X}$, does there exist a subvariety $Y \subseteq X$ tropicalizing to $\mathcal{Y}$?
	\end{problem}
	
	There are many situations in which the above problem has a negative answer. One of the most famously known is when considering tropical lines contained in a tropical cubic surface. From the Cayley-Salmon theorem we know that a cubic surface contains just $27$ lines, while there are examples of tropical cubic surfaces that contain infinitely many tropical lines \cite{panizzut2022tropical}.
	
	
	Even in the case when $\mathcal{X}$ and $\mathcal{Y}$ in Problem \ref{problem: relative realizability} are linear, the answer might still be negative. An example of this was provided in \cite{lamboglia2022tropical} in the context of tropical Fano schemes, and further studied in \cite{jell2022moduli}. Our main theorem offers another perspective on this phenomena (see Section \ref{sec: linear case} for more details). Roughly speaking, in a similar way as a tropical linear space has an underlying matroid, an inclusion of tropical linear spaces has an underlying matroid quotient, and the corresponding relative realizability problem is related to the realizability problem of the matroid quotient. We also make some remarks about the nonlinear case in Section \ref{sec: nonlinear case}.
	
	This paper is organized as follows. In Section \ref{sec: preliminaries} we recall some basic notions about matroid quotients and review factorizations and majors, in Section \ref{sec: realizability of quotients}  we prove our main result (Theorem \ref{thm: characterization of realizability of quotients}). In Section \ref{sec: applications}  we discuss some applications to the relative realizability problem in tropical geometry, and finally in Section \ref{sec: questions} we list some questions.
	
	\paragraph{Acknowledgements} I would like to thank Diane Maclagan, Bernd Sturmfels and Alex Fink for useful discussions and email exchanges. A special thank to Alheydis Geiger for insights about the relative realizability problem for lines on a cubic surface and to Hamdi Dërvodeli for sharing his notes on the Veronese embedding. Further thanks to Hannah Markwig and Victoria Schleis for fruitful exchanges that partly inspired this work.
	
	\section{Preliminaries}\label{sec: preliminaries}
	
	We assume that the reader is familiar with the basics of matroid theory \cite{oxley2011}.
	
	\subsection{Morphisms, strong maps and quotients}
	
	Let $M_1$ and $M_2$ be matroids on the ground sets $E_1$ and $E_2$ respectively.
	
	\begin{definition}
		 A \emph{morphism} of matroids $f:M_1 \rightarrow M_2$ is a function of sets $f:{E_1} \rightarrow {E_2}$ such that for every flat $F$ of $M_2$, $f^{-1}(F)$ is a flat of $M_1$.
	\end{definition}

	For a set $E$ and an element $o$ not in $E$, we will denote $E \cup \{o\}$ by $E_o$. Let $o$ be an element not in $E_1 \cup E_2$, and let ${E_i}_o$ be the ground set of $M_i \oplus U_{0,1}$ for $i \in \{1,2\}$.

	\begin{definition}
		A \emph{strong map} of matroids $f:M_1 \rightarrow M_2$ is a morphism of matroids $f_o: M_1 \oplus U_{0,1} \rightarrow M_2 \oplus U_{0,1}$ such that $f_o(o) = o$.
	\end{definition}
	
	
	From the definition, it is easy to see that the composition of two strong maps is a strong map. Therefore, matroids together with strong maps form a category. This category was studied in \cite{heunen2018category}.
		
	There are two natural examples of strong maps of matroids:
	\begin{itemize}
		\item \textbf{Extensions}: if $M$ is a matroid on $E$, and $N$ is an extension of $M$ on $E \cup S$, then the inclusion map $E_o \subseteq (E \cup S)_o$ is a strong map from $M$ to $N$.
		
		\item \textbf{Contractions}: if $I \subseteq E$, then the map $f: E_o \rightarrow E_o \setminus I$ defined by $f(o) = o$ and for every $e \in E$
		\[ f(e) = \begin{cases}
			e & \text{if } e \notin I, \\
			o & \text{if } e \in I,
		\end{cases} \]
		is a strong map from $M$ to $M/I$.
	\end{itemize}
	
	\begin{definition}
		If $M_1$ and $M_2$ have the same ground set $E$, and the identity on $E$ is a morphism of matroids, then $M_2$ is a \emph{quotient} of $M_1$, and the identity map on $E_o$ is the \emph{quotient map} from $M_1$ to $M_2$.
	\end{definition}
	\noindent
	Quotients will be denoted by $M_1 \twoheadrightarrow M_2$ or by $f:M_1 \twoheadrightarrow M_2$ to emphasize the underlying quotient map.
	
	As explained in \cite[Section 8.1]{Kung1986bookchapter}, in order to study strong maps of matroids, often times one can essentially restrict to study matroids quotients.

	\subsection{Realizability}\label{sec: realizability}
	
	\begin{definition}
		A rank $r$ matroid $M$ on $E$ is \emph{realizable} (or \emph{representable}) over a field $K$ if there exists a $K$-vector space $V$ of dimension $r$ and a map $\varphi:E \rightarrow V$ such that $A \subseteq E$ is independent on $M$ if and only if the vectors $\varphi(A)$ are linearly independent.
	\end{definition}

	Let $M$ be a matroid realizable over $K$, and let $\varphi_i : E \rightarrow V_i$ for $i \in \{ 1,2 \}$ be two maps realizing $M$. We say that $\varphi_1$ and $\varphi_2$ are \emph{equivalent} if there exists an isomorphism $\psi:V_1 \rightarrow V_2$ such that $\varphi_2 = \psi \circ \varphi_1$. A \emph{realization} of $M$ over $K$ is an equivalence class of such maps.
	
	If $E$ has cardinality $n$, a map $\varphi:E \rightarrow V$ realizing $M$ is often regarded as an $r \times n$ matrix $A$ with coefficients in $K$ once a basis of $V$ is fixed. Thus, an equivalence class is obtained by multiplying $A$ on the left by all invertible $r \times r$ matrices. Such an equivalence class can be identified with the $r$-dimensional vector subspace of $K^n$ generated by the rows of $A$. This can be also thought as a point in the Grassmannian $\Gr(r,n)$, which can be given by its Pl\"{u}cker coordinates, via the Pl\"{u}cker embedding of $\Gr(r,n)$ into $\mathbb{P}^{\binom{n}{r}-1}$.
	
	\begin{definition}
		Let $M_i$ be a matroid on $E_i$ for $i \in \{ 1,2 \}$. A morphism of matroids $f:M_1 \rightarrow M_2$ is \emph{realizable} over a field $K$ if both $M_1$ and $M_2$ are realizable over $K$, and there exists a map $\varphi_i:E_i \rightarrow V_i$ realizing $M_i$ for each $i \in \{ 1,2 \}$ and a linear map $g: V_1 \rightarrow V_2$ such that the following diagram is commutative:
		\[ \begin{tikzcd}
		E_1 \arrow[d, "f"] \arrow[r, "\varphi_1"] & V_1 \arrow[d, "g"] \\
		E_2 \arrow[r, "\varphi_2"] & V_2
		\end{tikzcd} \]
		A strong map (or quotient) of matroids is \emph{realizable} over a field $K$ if so is as a morphism of matroids.
	\end{definition}

	Assume that the morphism of matroids $f:M_1 \rightarrow M_2$ is realized over $K$ by the following two commutative diagrams:
	\[
	\begin{tikzcd}
	E_1 \arrow[d, "f"] \arrow[r, "\varphi_1"] & V_1 \arrow[d, "g"] \\
	E_2 \arrow[r, "\varphi_2"] & V_2
	\end{tikzcd}
	\qquad
	\begin{tikzcd}
	E_1 \arrow[d, "f"] \arrow[r, "\varphi_1'"] & V_1' \arrow[d, "g'"] \\
	E_2 \arrow[r, "\varphi_2'"] & V_2'
	\end{tikzcd}
	\]
	The above commutative diagrams realizing $f$ are \emph{equivalent} if there exist two isomorphisms $\psi_i : V_i \rightarrow V_i'$ for $i \in \{ 1,2 \}$ such that the following diagram is commutative:
	\[
	\begin{tikzcd}[row sep=scriptsize, column sep=scriptsize]
		& E_1 \arrow[dl, "\varphi_1"'] \arrow[dd, "f", near end] \arrow[dr, "\varphi_1'"] & \\
		V_1 \arrow[rr, crossing over, "\psi_1", "\sim"', near start] \arrow[dd, "g"'] & & V_1' \arrow[dd, "g'"] \\
		& E_2 \arrow[dl, "\varphi_2"'] \arrow[dr, "\varphi_2'"] & \\
		V_2 \arrow[rr, "\psi_2", "\sim"', near start] & & V_2' \\
	\end{tikzcd}
	\]
	A \emph{realization} of $f:M_1 \rightarrow M_2$ over $K$ is an equivalence class of such commutative diagrams.
	
	For example, if $E_1$ and $E_2$ are the same set with cardinality $n$, and the matroid $M_i$ has rank $r_i$ for $i \in \{ 1,2 \}$, a realization of the quotient $M_1 \twoheadrightarrow M_2$ can be thought to be a realization of $M_i$ given by an $r_i$-dimensional $K$-vector space $U_i \subseteq K^n$ for $i \in \{ 1,2 \}$ (in the sense we have described above) such that $U_2 \subseteq U_1$. Therefore, a realization of $M_1 \twoheadrightarrow M_2$ can be thought of as a point in the flag variety $\Fl(r_1,r_2;n)$, which can be given by its Pl\"{u}cker coordinates once we embed the flag variety in a product of Grassmannians $\Fl(r_1,r_2;n) \subseteq \Gr(r_1,n) \times \Gr(r_2,n) \subseteq \mathbb{P}^{\binom{n}{r_1}-1} \times \mathbb{P}^{\binom{n}{r_2}-1}$.
	
	\begin{definition}
		A \emph{diagram of matroids} consists of a set of matroids $\{M_\lambda\}_{\lambda \in \Lambda}$ and a set of some morphisms between them $\{ f: M_\lambda \rightarrow M_{\lambda'} \}$.
	\end{definition}

	The previous definition can be recast in terms of category theory. Denote by $\mathbf{Mat}$ the category of matroids with their morphisms. A diagram of matroids can be thought as a functor $D: \mathbf{I} \rightarrow \mathbf{Mat}$, where $\mathbf{I}$ is some index category.
	
	\begin{definition}\label{def: realization diagram matroid}
		A diagram of matroids $(\{M_\lambda\}_{\lambda \in \Lambda}, \{ f: M_\lambda \rightarrow M_{\lambda'} \})$ is \emph{realizable} over a field $K$ if for every $\lambda \in \Lambda$ there exists a map $\varphi_\lambda: E_\lambda \rightarrow V_\lambda$ that realizes $M_\lambda$, and for every morphism $f: M_\lambda \rightarrow M_{\lambda'}$ in the diagram, there exists a linear map $g: V_\lambda \rightarrow V_{\lambda'}$ such that the following diagram commutes:
		\[ \begin{tikzcd}
		E_\lambda \arrow[d, "f"] \arrow[r, "\varphi_\lambda"] & V_\lambda \arrow[d, "g"] \\
		E_{\lambda'} \arrow[r, "\varphi_{\lambda'}"'] & V_{\lambda'}
		\end{tikzcd} \]
		Two such family of maps $(\{\varphi_{i,\lambda}: E_\lambda \rightarrow V_{i,\lambda} \}_{\lambda \in \Lambda},\{g_i:V_{i,\lambda} \rightarrow V_{i,\lambda'} \})$ for $i \in \{ 1,2 \}$ are \emph{equivalent} if for every $\lambda \in \Lambda$ there exists an isomorphism $\psi_{\lambda} : V_{1,\lambda} \rightarrow V_{2,\lambda}$ such that, for every morphism $f:M_\lambda \rightarrow M_{\lambda'}$ in the diagram of matroids, the following diagram commutes:
		\[
		\begin{tikzcd}[row sep=scriptsize, column sep=scriptsize]
			& E_\lambda \arrow[dl, "\varphi_{1,\lambda}"'] \arrow[dd, "f", near end] \arrow[dr, "\varphi_{2,\lambda}"] & \\
			V_{1,\lambda} \arrow[rr, crossing over, "\psi_\lambda", "\sim"', near start] \arrow[dd, "g_1"'] & & V_{2,\lambda} \arrow[dd, "g_2"] \\
			& E_{\lambda'} \arrow[dl, "\varphi_{1,\lambda'}"', near start] \arrow[dr, "\varphi_{2,\lambda'}", near start] & \\
			V_{1,\lambda'} \arrow[rr, "\psi_{\lambda'}", "\sim"', near start] & & V_{2,\lambda'} \\
		\end{tikzcd}
		\]
		A \emph{realization} of a diagram of matroids is an equivalence class of families of maps $(\varphi_{i,\lambda},g_i)$ as above.
	\end{definition}
	
	The above definitions can be more compactly recast in terms of category theory. We will follow a similar point of view taken in \cite[Remark 2.1]{eur2020logarithmic}.
	
	Let $\mathbf{Mat}(K)$ be the category in which an object is a map $\varphi: E \rightarrow V$ realizing a matroid $M(\varphi)$ on the ground set $E$ over a field $K$, and a morphism from $\varphi_1$ to $\varphi_2$ is a commutative diagram of the form
	\[ \begin{tikzcd}
	E_1 \arrow[d, "f"] \arrow[r, "\varphi_1"] & V_1 \arrow[d, "g"] \\
	E_2 \arrow[r, "\varphi_2"'] & V_2
	\end{tikzcd} \]
	realizing a morphism $f:M(\varphi_1) \rightarrow M(\varphi_2)$. From the definition, we have a functor $\mathcal{F}:\mathbf{Mat}(K) \rightarrow \mathbf{Mat}$ defined by $\varphi \mapsto M(\varphi)$.
	
	Now let $D: \mathbf{I} \rightarrow \mathbf{Mat}$ be a diagram of matroids. Without loss of generality, we can assume that $D$ is injective on objects. Thus the image of $D$ is a subcategory $\im(D)$ of $\mathbf{Mat}$. Then, the diagram $D$ is realized over a field $K$ by a functor $F: \im(D) \rightarrow \mathbf{Mat}(K)$ such that $\mathcal{F} \circ F$ is the identity on $\im(D)$. Two such functors realizing $D$ are equivalent if they are naturally isomorphic, and a realization of $D$ is an isomorphism class of such naturally isomorphic functors.
	
	\begin{question}
		Let $D: \mathbf{I} \rightarrow \mathbf{Mat}$ be a diagram of matroids. Does the notion of realizability of Definition \ref{def: realization diagram matroid} give rise to a stratification of the quiver Grassmannian corresponding to $D$ (defined as in \cite{schofield1992general})? 
	\end{question}
	
	\subsection{Modular cuts and weak maps}
	
	Following \cite[Section 7.3]{Kung1986bookchapter}, we now recall the notions of one element extensions of a matroid and modular cuts.
	
	An \emph{extension} of a matroid $M$ on $E$ is a matroid $N$ on $E \cup S$ such that $N \setminus S = M$. If $S = \{ e \}$ then $N$ is said a \emph{one element extension} of $M$, denoted $M+e$. In order to define a one element extension $M+e$ of $M$, we need just to specify, for each flat $F$ of $M$, whether or not $F$ and $F \cup \{ e \}$ are flats of $N$. Thus, we can partition the family $\mathscr{F}(M)$ of flats of $M$ into three families:
	\begin{align*}
	\mathscr{F}_1 &= \{ F \in \mathscr{F}(M): F \in \mathscr{F}(M+e), F \cup \{ e \} \in \mathscr{F}(M+e) \}, \\
	\mathscr{F}_2 &= \{ F \in \mathscr{F}(M): F \in \mathscr{F}(M+e), F \cup \{ e \} \notin \mathscr{F}(M+e) \}, \\
	\mathscr{F}_3 &= \{ F \in \mathscr{F}(M): F \notin \mathscr{F}(M+e), F \cup \{ e \} \in \mathscr{F}(M+e) \}.
	\end{align*}
	
	Now $\mathscr{F}_1$ is downward closed, i.e. any flat contained in a flat in $\mathscr{F}_1$ is in $\mathscr{F}_1$. On the other hand, $\mathscr{F}_3$ is upward closed, i.e. any flat that contains a flat in $\mathscr{F}_3$ is in $\mathscr{F}_3$. Finally, a flat is in $\mathscr{F}_2$ if and only if it is not in $\mathscr{F}_3$ and it is covered by a flat in $\mathscr{F}_3$. In light of the previous statements, this partition is uniquely determined by $\mathscr{F}_3$. Families of flats of type $\mathscr{F}_3$ are called \emph{modular cuts}, and they can be characterized as follows. A family $\mathcal{M}$ of flats of $M$ is a modular cut of $M$ if and only if
	\begin{itemize}
		\item $\mathcal{M}$ is upward closed and nonempty,
		\item if $A,B \in \mathcal{M}$ are such that $\rk(A) + \rk(B) = \rk(A \cap B) + \rk(A \cup B)$, then $A \cap B \in \mathcal{M}$.
	\end{itemize}
	
	It is now clear that, the one element extensions of $M$ are one to one with modular cuts of $M$. We will denote by $M+_{\mathcal{M}}e$ the one element extension of $M$ corresponding to the modular cut $\mathcal{M}$. 
	
	\begin{definition}
		Let $M_i$ be a matroid on $E_i$ for $i \in \{ 1,2 \}$. A \emph{weak map} is a function of sets $f: E_1 \rightarrow E_2$ such that if $I$ is an independent set of $M_2$, then $f^{-1}(I)$ is an independent set of $M_1$.
	\end{definition}

	Weak maps provide the following partial order on matroids on a common ground set $E$: if the identity on $E$ is a weak map from $M_1$ to $M_2$ then $M_1 \leq M_2$.
	
	\begin{lemma}[{\cite[Section 2.6]{kennedy1975majors}}]\label{lem: modular cuts and weak order}
		Let $M$ be a matroid on $E$, and let $\mathcal{M}_1$ and $\mathcal{M}_2$  be two modular cuts of $M$. The following are equivalent:
		\begin{enumerate}
			\item $\mathcal{M}_1 \subseteq \mathcal{M}_2$,
			\item $M +_{\mathcal{M}_1} e \leq M +_{\mathcal{M}_2} e$,
			\item $(M +_{\mathcal{M}_1} e) / e \leq (M +_{\mathcal{M}_2} e) / e$.
		\end{enumerate}
	\end{lemma}
	
	\subsection{Factorizations and majors}
	
	
	In this section we recall the basics of factorizations and majors of matroid quotients. Our main reference here is \cite{kennedy1975majors}.
	
	The \emph{nullity} of a morphism of matroids $f: M_1 \rightarrow M_2$ (or strong map, or quotient) is the integer $n(f) = \rk(M_1) - \rk(M_2)$. Quotients $M_1 \twoheadrightarrow M_2$ of nullity one are called \emph{elementary}. In this case, $M_2$ is said to be an \emph{elementary quotient} of $M_1$, and $M_2$ can be obtained from $M_1$ by a one element extension followed by a contraction of the element added. Since, as we have seen in the previous section, one element extensions correspond to modular cuts, so do elementary quotients.
	
	In the remainder of this section, if not stated otherwise, $f: M_1 \twoheadrightarrow M_2$ is a quotient of nullity $k$ of matroids sharing a common ground set $E$ of cardinality $n$.
	
	\begin{definition}
		A \emph{factorization} of a quotient $f:M_1 \twoheadrightarrow M_2$ is a sequence of matroids $N =(N_0,N_1,\dots,N_k)$ such that we have a chain of elementary quotients as follows:
		\begin{equation}\label{eq: factorization}
			M_1 = N_0 \twoheadrightarrow N_1 \twoheadrightarrow \dots \twoheadrightarrow N_{k-1} \twoheadrightarrow N_k = M_2.
		\end{equation}
		The \emph{length} (or \emph{nullity}) of a factorization is equal to the nullity of its quotient.
	\end{definition}
	
	The nullity of a subset $A \subseteq E$ is $n(A) = \rk_{M_1}(A) - \rk_{M_2}(A)$. From \cite[Proposition 8.1.6]{Kung1986bookchapter} the nullity is an increasing set function. We will frequently use this fact in the rest of the paper.
	
	\begin{proposition definition}[Higgs lift {\cite[Lemma 8.1.6]{Kung1986bookchapter}}]\label{prop-def: Higgs lifts}
		Given a quotient $f: M_1 \twoheadrightarrow M_2$, the flats of $M_2$ together with the flats of $M_1$ with nullity strictly less than $i$	are the flats of a matroid $L^i(f)$, called the \emph{$i$-th Higgs lift} of $f$.
	\end{proposition definition}
	
	The Higgs lifts of $f$ form a chain of elementary quotients:
	\[ M_1 = L^k(f) \twoheadrightarrow L^{k-1}(f) \twoheadrightarrow \dots \twoheadrightarrow L^1(f) \twoheadrightarrow L^0(f) = M_2. \]
	Therefore, the sequence $(L^i(f))_i$ is a factorization of $f$, called the \emph{Higgs factorization}. Hence, every quotient has at least one factorization, but in general factorizations are not unique (see \cite{kennedy1975majors}).

	Factorizations are particular instances of diagrams of matroids. Thus we can consider realizability and realizations of factorizations in the sense of Definition \ref{def: realization diagram matroid}.
	
	\begin{definition}
		A \emph{major} of a quotient $f:M_1 \twoheadrightarrow M_2$ of nullity $k$ is a matroid $H$ on the ground set $E \cup K$, where $K$ is of cardinality $k$ and disjoint from $E$, such that $M_1 = H \setminus K$ and $M_2 = H / K$.
	\end{definition}
	
	In other words, if $H$ is a major of $f:M_1 \twoheadrightarrow M_2$, then $f$ can be factored into an extension followed by a contraction: $M_1 \rightarrow H \rightarrow M_2$.
	
	\begin{theorem}[Higgs \cite{higgs1968strong}]\label{thm: Higgs factorization theorem}
		Every matroid quotient admits a major.
	\end{theorem}

	The major constructed by Higgs in the proof of the above theorem is called \emph{Higgs major}. The Higgs major of $f:M_1 \twoheadrightarrow M_2$ is constructed by taking the $k$-th Higgs lift of the quotient $M_1 \oplus U_{k,k} \twoheadrightarrow M_2 \oplus U_{0,k}$. See \cite[Theorem 8.2.7]{Kung1986bookchapter} for more details.
	
	In general, majors of a quotient are not unique. See \cite[Example 3.14]{kennedy1975majors} for an example of a quotient  admitting two distinct majors.
	
	\begin{lemma}[{ \cite[Lemma 3.3-3.4]{kennedy1975majors}}]\label{lem: standard extension contraction lemma}
		Let $M$ be a matroid on $E$, $e \in E$ and $f$ an element not in $E$.
		\begin{enumerate}
			\item If $\mathcal{M}$ is a modular cut of $M / e$, then $\mathcal{M} \cup e = \{ F \cup \{e\} : F \in \mathcal{M} \}$ is a modular cut of $M$.
			\item $(M/e) +_{\mathcal{M}} f = (M +_{\mathcal{M} \cup e} f) /e$.
		\end{enumerate}
	\end{lemma}
	
	We now introduce some more notation: arrows of type $\overset{e}{\underset{\mathcal{M}}{\longrightarrow}}$ denote a strong map arising from a one element extension by an element $e$ with respect to the modular cut $\mathcal{M}$; arrows of type $\overset{e}{\longrightarrow}$ denote the strong map arising from the quotient by the element $e$. With this notation, the above lemma can be rephrased simply by saying that the following diagram of matroids and strong maps, called the \emph{standard extension-contraction diamond} in \cite{kennedy1975majors}, commutes:
	
	\begin{equation}\label{eq: standard extensions contraction diamond}
		\begin{tikzcd}
		& M_2 \arrow[dr, "e"] & \\
		M_1 \arrow[ur, "f", "\mathcal{M} \cup e"'] \arrow[dr, "e"] & & M_3 \\
		& M_4 \arrow[ur, "f", "\mathcal{M}"']
		\end{tikzcd}
	\end{equation}
	
	Now let $N = (N_0, \dots, N_k)$ be a factorization of a quotient $f:M_1 \twoheadrightarrow M_2$ of nullity $k$.  Every elementary quotient $N_{i-1} \twoheadrightarrow N_{i}$ can be factored into an extension by one element $e_i$ with modular cut $\mathcal{M}_i$ followed by a contraction by $e_i$. In particular, this means that we can go from $M_1 = N_0$ to $M_2 = N_k$ by extending and then contracting $e_i$ for $i$ from $1$ to $k$. By repeatedly applying Lemma \ref{lem: standard extension contraction lemma}, we can first extend $M_1$ by $e_1,e_2,e_3,\dots,e_k$ with respect to the modular cuts $\mathcal{M}_1,\mathcal{M}_2 \cup e_1, \mathcal{M}_3 \cup e_1 \cup e_2, \dots, \mathcal{M}_k \cup e_1 \cup \dots \cup e_{k-1}$ respectively, and then contract by $e_1,\dots,e_k$ all at once. Thus, the matroid we obtain after all such extensions is, by definition, a major of $f$. This major will be denoted by $\mathfrak{M}(N)$, and clearly it depends on the factorization $N$ we started with.
	
	The above procedure can be better displayed by constructing a triangular commutative diagram associated with the factorization. Starting with the extension and contraction maps along the base of the triangle, each successive level is constructed by applying Lemma \ref{lem: standard extension contraction lemma}. We illustrate this procedure with the following triangular commutative diagram of a factorization of a quotient of nullity $3$.
	
	\begin{equation}\label{eq: triangular diagram}
		\begin{tikzcd}
		&&& N_{3,0} \arrow[dr, "e_1"] &&& \\
		&& N_{2,0} \arrow[ur, "e_3", "\mathcal{M}_3 \cup e_1 \cup e_2"'] \arrow[dr, "e_1"] && N_{2,1} \arrow[dr, "e_2"] && \\
		& N_{1,0} \arrow[ur, "e_2", "\mathcal{M}_2 \cup e_1"'] \arrow[dr, "e_1"] && N_{1,1} \arrow[ur, "e_3", "\mathcal{M}_3 \cup e_2"'] \arrow[dr, "e_2"] && N_{1,2} \arrow[dr, "e_3"] & \\
		N_{0,0} \arrow[ur, "e_1", "\mathcal{M}_1"'] \arrow[rr, two heads] && N_{0,1} \arrow[ur, "e_2", "\mathcal{M}_2"'] \arrow[rr, two heads] && N_{0,2} \arrow[ur, "e_3", "\mathcal{M}_3"'] \arrow[rr, two heads] && N_{0,3}
		\end{tikzcd}
	\end{equation}
	
	We now wish to construct a factorization from a major. Let $H$ be a major of $f:M_1 \twoheadrightarrow M_2$ on the ground set $E \cup K$, where $K = \{e_1,\dots,e_k\}$. We define the factorization $\mathfrak{F}(H) = (\mathfrak{F}(H)_i)$ of $f$ by
	\[ \mathfrak{F}(H)_i = \big( H / \{e_1,\dots,e_i\} \big) \setminus \{e_{i+1},\dots,e_k\}.  \]
	Let $\mathscr{F}_f$ denote the set of factorizations of the quotient $f$, and let $\mathscr{M}_f$ denote the set of majors of $f$. Thus far, we have defined the following two maps:
	\begin{align*}
		& \mathfrak{M}: \mathscr{F}_f \rightarrow \mathscr{M}_f, \\
		& \mathfrak{F}: \mathscr{M}_f \rightarrow \mathscr{F}_f.
	\end{align*}
	
	\begin{theorem}[Kennedy \cite{kennedy1975majors}]
		~
		\begin{enumerate}
			\item The map $\mathfrak{F}$ is surjective but not injective, whereas $\mathfrak{M}$ is injective but not surjective.
			\item The map $\mathfrak{F} \circ \mathfrak{M}$ is the identity on $\mathscr{F}_f$, whereas $\mathfrak{M} \circ \mathfrak{F}$ is a coclosure operator on $\mathscr{M}_f$ with the partial order given by the identity on the ground set being a weak map.
			\item The map $\mathfrak{F}$ applied to the Higgs major gives us the Higgs factorization and conversely $\mathfrak{M}$ applied to the Higgs factorization gives us the Higgs major.
		\end{enumerate}
	\end{theorem}

	
	\section{Realizability of quotients of matroids}\label{sec: realizability of quotients}

	\begin{lemma}
		Every elementary quotient $f:M_1 \twoheadrightarrow M_2$ has a unique major.
	\end{lemma}
	\begin{proof}
		Existence follows from Theorem \ref{thm: Higgs factorization theorem}. For uniqueness, let $H_1$ and $H_2$ be two majors of $f$. Then $H_1$ and $H_2$ are obtained by extending $M_1$ with an element $e$ with respect to the modular cuts $\mathcal{M}_1$ and $\mathcal{M}_2$ respectively. From Lemma \ref{lem: modular cuts and weak order} we have that $H_1 \leq H_2$ if and only if $H_1 / e \leq H_2 / e$. Since $H_1/e = H_2/e = M_2$, it follows that $H_1 = H_2$.
	\end{proof}
	
	There are several results in the literature that are similar or implicitly related to the next result. See for instance \cite[Section 5.1]{brandt2021tropical} and \cite[Section 2.2]{jell2022moduli}.
	
	\begin{proposition}\label{prop: realization elementary quotient}
		Let $f:M_1 \twoheadrightarrow M_2$ be an elementary quotient. The realizations of $f$ are in one-to-one correspondence with the realizations of its unique major.
	\end{proposition}
	\begin{proof}
		Let $H$ be the major of $f$ on the ground set $E \cup e$, for some $e \notin E$.  Let $\varphi: E_o \cup e \rightarrow W$ be a realization of $H$, and let $U$ be the space generated by $\varphi(e)$. Then a realization of $f$ is given by the following commutative diagram
		\[ \begin{tikzcd}
			E_o \arrow[r, "\varphi_1"] \arrow[d,"f"] & W \arrow[d,"g"] \\
			E_o \arrow[r,"\varphi_2"] & W / U
		\end{tikzcd}  \]
		where $\varphi_1$ is the restriction to $E_o$ of $\varphi$, and $\varphi_2$ is the composition of $\varphi_1$ and the quotient map by $U$. Conversely, consider a realization of $f$. We can assume that such a realization is represented by a diagram as above, where again the space $U$ is one dimensional, since $f$ is elementary. Thus, $U$ is generated by some nonzero vector $v$, and the map $\varphi: E_o \cup e \rightarrow W$ defined by $\varphi_{|E_o} = \varphi_1$ and $\varphi(e) = v$ is a realization of $H$.
		
		Now suppose that $\psi_i : E_o \cup e \rightarrow W_i$ for $i \in \{1,2\}$ represent two realizations of $H$ such that they give rise to the same realization of $f$. This means that we have a commutative diagram
		\[
		\begin{tikzcd}[row sep=scriptsize, column sep=scriptsize]
			& E_o \arrow[dl, "\varphi_{1}"'] \arrow[dd, "f", near end] \arrow[dr, "\varphi_{2}"] & \\
			W_1 \arrow[rr, crossing over, "h_1", "\sim"', near start] \arrow[dd, "/U_1"'] & & W_2 \arrow[dd, "/U_2"] \\
			& E_o \arrow[dl, "/U_1 \circ \varphi_1"', near start] \arrow[dr, "/U_2 \circ \varphi_2", near start] & \\
			W_1/U_1 \arrow[rr, "h_2", "\sim"', near start] & & W_2/U_2 \\
		\end{tikzcd}
		\]
		where $U_i$ is the space generated by $\psi_i(e)$, $\varphi_i= {\psi_i}_{|E_o}$, and $h_i$ is an isomorphism, for $i \in \{1,2\}$. Since the diagram is commutative, $h_1$ sends $U_1$ to $U_2$. Further, $U_1$ and $U_2$ are one-dimensional, so $U_i$ is generated by some nonzero vector $v_i$ for $i \in \{1,2\}$. Hence, $h_1$ sends $v_1$ to a scalar multiple of $v_2$, therefore $\psi_1$ and $\psi_2$ represent the same realization of $H$.
	\end{proof}

	\begin{lemma}\label{lem: standard extension contraction realizability lemma}
	 	Consider the standard extension-contraction diamond as in \eqref{eq: standard extensions contraction diamond} and let $K$ be an infinite field. Then, $M_2$ is realizable over $K$ if and only if the subdiagram $M_1 \xrightarrow{e} M_4 \overset{f}{\underset{\mathcal{M}}{\longrightarrow}} M_3$ obtained from \eqref{eq: standard extensions contraction diamond} is realizable over $K$.
	\end{lemma}
	\begin{proof}
		Since $M_1$, $M_3$ and $M_4$ are minors of $M_2$, if the latter is realizable, then it is clear that the lower part of the diagram is realizable.
		
		Conversely, assume the lower part of the diagram is realizable. Let $E$ be the cardinality $n$ ground set of $M_1$, so that $E \cup f$, $E \cup f \setminus e$ and $E \setminus e$ are the ground sets of $M_2$, $M_3$ and $M_4$ respectively. If $e$ is a loop of $M_1$, then it is also a loop of $M_2$, and contracting $e$ is the same as deleting $e$. Thus a realization of $M_2$ is given by extending the realization of $M_3$ by sending $e$ to the zero vector. Therefore, we can assume that $e$ is not a loop of $M_1$ and part of a basis. In addition, if $f$ is a coloop of $M_3$, this will imply $\mathcal{M} = \emptyset$, so we will have $\mathcal{M} \cup e = \emptyset$ as well, that is, $f$ will be a coloop also of $M_2$ and a realization of $M_2$ can be constructed by extending a realization $\varphi:E \rightarrow V$ of $M_1$ to $E \cup f$ by sending $f$ to $\underline{0} \oplus 1 \in V \oplus K$. Therefore we can assume that $f$ is not a coloop of $M_3$. Thus, if $M_1$ is of rank $r$ then $M_2$ has rank $r$ as well and $M_3$ and $M_4$ have rank $r-1$.
		
		Now, up to equivalence, the diagram $M_1 \xrightarrow{e} M_4 \overset{f}{\underset{\mathcal{M}}{\longrightarrow}} M_3$ is realized by a commutative diagram of the following form:
		\begin{equation*}
			\begin{tikzcd}
				E_o \arrow[r] \arrow[d, "\varphi_1"] & E_o \setminus e \arrow[r] \arrow[d, "\varphi_2"] & E_o \cup f \setminus e \arrow[d,"\varphi_3"] \\
				K^r \arrow[r, "g"] & K^{r-1} \arrow[r, "id_{K^{r-1}}"] & K^{r-1}
			\end{tikzcd}
		\end{equation*}
		where $\varphi_1(e) = e_1 \in K^r$, $g: K^r \rightarrow K^{r-1}$ is the quotient map by $e_1$ and $\varphi_3(f) = (f_2,\dots,f_{r}) \in K^{r-1}$. Now a realization of $M_2$ is given by a function $\varphi: E_o \cup f \rightarrow K^{r}$ defined by $\varphi_{|E_o} = \varphi_1$ and $\varphi(f) = (f_1,f_2,\dots,f_r) \in K^r$, where the element $f_1 \in K$ is chosen such that the vector $\varphi(f)$ is in the subspace of $K^r$ generated by $\varphi_1(F)$, for $F$ a flat of $M_1$, if and only if $F \in \mathcal{M} \cup e$. In the remainder of the proof, we will show that such a choice is possible.
		
		First, if $F$ is a flat of $M_1$ in $\mathcal{M} \cup e$, then $e \in F$ and $F \setminus e \in \mathcal{M}$ by definition. Thus, $(f_2,\dots,f_r)$ belongs to the space generated by $\varphi_2(F \setminus e)$, and, since $\varphi_1(e) = e_1$, this implies that $(f_1,\dots,f_r) \in \langle \varphi_1(F) \rangle$ for every $f_1 \in K$. Now let $F$ be a flat of $M_1$ such that the space $L = \langle \varphi_1(F) \rangle$ contains the affine line $A = \{ (x,f_2,\dots,f_r) \in K^r: x \in K \}$. This means that $L$ contains $e_1$, so the flat $F$ contains $e$. In particular, the image of $L$ under the quotient map $g$ contains the vector $(f_2,\dots,f_r)$, and this implies that $F \setminus e \in \mathcal{M}$, that is $F \in \mathcal{M} \cup e$. Thus, for any flat $F$ of $M_1$ not in $\mathcal{M} \cup e$, the space $L = \langle \varphi_1(F) \rangle$ does not contain the affine line $A$, hence it intersects $A$ in at most one point $p_F$. Now it is enough to chose $f_1$ such that $\varphi(f) \in A \setminus \{ p_F : F \notin \mathcal{M} \cup e \}$. This last set is not empty since the points $p_F$ are finitely many, as the flats of $M_1$ are finite, while the affine space $A$ contains infinitely many points since $K$ is infinite. \qedhere

	\end{proof}
	
	The above result is easily seen to be not true for finite fields, as the following example shows.
	
	\begin{example}
		Let $K = \{ 0,1 \}$ be the field with two elements. Consider the following standard extension-contraction diamond:
		\begin{equation*}
			\begin{tikzcd}
				& U_{2,4} \arrow[dr, "e"] & \\
				U_{2,3} \arrow[ur, "f", "\mathcal{M} \cup e"'] \arrow[dr, "e"] & & U_{1,3} \\
				& U_{1,2} \arrow[ur, "f", "\mathcal{M}"']
			\end{tikzcd}
		\end{equation*}
		Up to permutations, the matroids $U_{2,3}$, $U_{1,3}$ and $U_{1,2}$ have each a unique realization over $K$. These realizations constitute also a realization of the diagram $U_{2,3} \xrightarrow{e} U_{1,3} \overset{f}{\underset{\mathcal{M}}{\longrightarrow}} U_{1,2}$ over $K$, while $U_{2,4}$ is not realizable over $K$.
	\end{example}
	
	\begin{corollary}\label{cor: factorization realizable iff major realizable}
		A factorization $N$ of a strong map $f:M_1 \twoheadrightarrow M_2$ is realizable over an infinite field $K$ if and only if so is its major $\mathfrak{M}(N)$.
	\end{corollary}
	\begin{proof}
		Let $\Delta(N)$ be the triangular diagram induced by $N$, constructed by following the procedure we illustrated in \eqref{eq: triangular diagram}. Then, the statement follows by applying Proposition \ref{prop: realization elementary quotient} along the base of the triangle, and by repeatedly applying Lemma \ref{lem: standard extension contraction realizability lemma} on all standard extension-contraction diamonds.
	\end{proof}
	
	\begin{lemma}\label{lem: flats covered by nullity + 1}
		Let $f:M_1 \twoheadrightarrow M_2$ be a quotient of matroids with common ground set $E$. A flat $F$ of $M_1$ is not closed in $M_2$ if and only if it is covered by a flat of $M_1$ of nullity $n(F)+1$.
	\end{lemma}
	\begin{proof}
		Since from \cite[Proposition 8.1.6]{Kung1986bookchapter} the nullity is a non-decreasing function, any flat that covers $F$ has nullity greater or equal than $F$. Suppose that $F$ is closed in $M_2$. Then, for every $a \in E \setminus F$ we have $r_{M_1}(F \cup a) = r_{M_1}(F) + 1$ and $r_{M_2}(F \cup a) = r_{M_2}(F) + 1$, thus $n(F \cup a) = n(F)$. Let $G$ be the closure of $F \cup a$ in $M_1$, then $r_{M_1}(G) = r_{M_1}(F \cup a)$, also from \cite[Proposition 8.1.6]{Kung1986bookchapter}, $G$ is contained in the closure of $F \cup a$ in $M_2$, so $r_{M_2}(G) =r_{M_2}(F \cup a)$. Hence $n(G) = n(F)$, so all the flats of $M_1$ that cover $F$ have the same nullity of $F$. On the other hand, if $F$ is not closed in $M_2$, then there exists $a \in E \setminus F$ such that $r_{M_2}(F \cup a) = r_{M_2}(F)$, and for the same reasoning above, we have $n(F \cup a) = n(F) + 1$, so the closure of $F \cup a$ in $M_1$ is a flat of $M_1$ of nullity $n(F)+1$ that covers $F$.
	\end{proof}
	
	\begin{lemma}\label{lem: modular cut of Higgs lift}
		Let $f: M_1 \twoheadrightarrow M_2$ be a quotient of matroids of nullity $k$ and let $(L^i(f))_i$ be its Higgs factorization. Then, the modular cut of $M_1$ associated to the elementary quotient $M_1 \twoheadrightarrow L^{k-1}(f)$ is given by all the nullity $k$ flats of $M_1$.
	\end{lemma}
	\begin{proof}
		From Lemma \ref{lem: flats covered by nullity + 1} all nullity $k$ flats of $M_1$ are closed in $M_2$. Therefore, from Proposition-Definition \ref{prop-def: Higgs lifts}, the flats of $L^{k-1}(f)$ consists of the flats of $M_1$ that are either closed in $M_2$ or do not have nullity $k-1$.  From Lemma \ref{lem: flats covered by nullity + 1}, this means that we are removing from $M_1$ all the flats covered by a nullity $k$ flat. This shows that the modular cut of this elementary quotient is given by the nullity $k$ flats.
	\end{proof}
	
	As an application of the above lemma, we compute the Higgs major of quotients of uniform matroids.
	
	\begin{example}[Higgs major of a quotient of uniform matroids]\label{ex: Higgs of uniform matroids}
		Consider the quotient $f: U_{r+k,n} \twoheadrightarrow U_{r,n}$ of uniform matroids on the common ground set $E$. The Higgs major of $f$ is $U_{r+k,n+k}$. We can prove this by induction on $k$. In fact, from Lemma \ref{lem: modular cut of Higgs lift} the modular cut of the elementary quotient $U_{r+k,n} \twoheadrightarrow L^{k-1}(f)$ is given by the nullity $k$ flats of $U_{r+k,n}$, which in this case those consist of just the set $E$. This means that the corresponding one element extension of $U_{r+k,n}$ give rise to the uniform matroid $U_{r+k,n+1}$. Now, $L^{k-1}(f)$ is the contraction of $U_{r+k,n+1}$ by the extended element, so we have $L^{k-1}(f) = U_{r+k-1,n}$. Now by induction, this means that the Higgs factorization of $f$ is $U_{r+k,n} \twoheadrightarrow U_{r+k-1,n} \twoheadrightarrow \dots \twoheadrightarrow U_{r+1,n} \twoheadrightarrow U_{r,n}$, and the Higgs major is $U_{r+k,n+k}$.
	\end{example}
	
	\begin{proposition}\label{prop: map induces factorization}
		Let $f: M_1 \twoheadrightarrow M_2$ be a quotient of matroids of nullity $k$. A realization of $f$ over an infinite field $K$ induces a realization of the Higgs factorization $(L^i(f))_i$ over $K$.
	\end{proposition}
	\begin{proof}
		Fix a realization of $f$ given by the following commutative diagram
		\[ \begin{tikzcd}
		E_o \arrow[r, "\varphi_1"] \arrow[d,"f"] & V_1 \arrow[d,"g"] \\
		E_o \arrow[r,"\varphi_2"] & V_2
		\end{tikzcd}  \]
		where $E$ is the common ground set of $M_1$ and $M_2$. Set $U = \ker g$, so we have $\dim U = k$. Up to equivalence, we can assume that $V_2 \simeq V_1 / U$ and $g$ is the quotient map by $U$. We proceed by induction on the nullity $k$ of $f$. If $k = 0$, then $M_1 = M_2$ and there is nothing to prove.
		Now assume $k > 0$. By Lemma \ref{lem: modular cut of Higgs lift} the modular cut of $M_1$ associated to the elementary quotient $M_1 \twoheadrightarrow L^{k-1}(f)$ is given by the family of flats of nullity $k$. Now, for every flat $F$ of $M_1$ we denote by $L_F$ the subspace of $V_1$ generated by $\varphi_1(F)$. The nullity of $F$ is equal to the dimension of $L_F \cap U$. In fact, we have
		\[ n(F) = r_{M_1}(F) - r_{M_2}(F)  = \dim(L_F) - \dim(L_F /  \big(L_F \cap U) \big) = \dim (L_F \cap U). \]
		Therefore, if the nullity of $F$ is $k$ then $L_F$ contains $U$. Now the set difference $W = U \setminus \bigcup_{n(F) = k-1} L_F$ is nonempty since the base field $K$ is infinite. Thus we can choose a vector $v \in W$. The vector $v$ is in $U$, thus it belongs to $L_F$ if $F$ is a flat of $M_1$ of nullity $k$, but it does not if $F$ has nullity $k-1$. Hence, the quotient map by $\langle v \rangle$ is a realization of $M \twoheadrightarrow L^{k-1}(f)$, and the quotient map from $V_1 / \langle v \rangle$ to $V_2$ by $U / \langle v \rangle$ is a realization of the quotient $L^{k-1}(f) \twoheadrightarrow M_2$. From \cite[Proposition 4.8]{kennedy1975majors} the sequence $(L^{k-1}(f),\dots,L^0(f))$ is the Higgs factorization of $L^{k-1}(f) \twoheadrightarrow M_2$. Thus, we can apply induction, so the realization of $L^{k-1}(f) \twoheadrightarrow M_2$ induces a realization of the factorization starting from $k-1$ to $0$. Further, note that in every step, we did not change the realizations of $M_1$ and $M_2$ we started with. Therefore, we can adjoin the realization of $M_1 \twoheadrightarrow L^{k-1}(f)$ with that of the rest of the factorization.
	\end{proof}

	\begin{theorem}\label{thm: characterization of realizability of quotients}
		Let $f:M_1 \twoheadrightarrow M_2$ be a quotient of matroids, and let $K$ be an infinite field. The following statements are equivalent:
		\begin{enumerate}
			\item $f$ is realizable over $K$,
			\item the Higgs factorization of $f$ is realizable over $K$,
			\item the Higgs major of $f$ is realizable over $K$.
		\end{enumerate}
	\end{theorem}
	\begin{proof}
		The implication $(1) \Rightarrow (2)$ is Proposition \ref{prop: map induces factorization} and $(2) \Rightarrow (1)$ is clear. Finally, $(2) \Leftrightarrow (3)$ follows from Corollary \ref{cor: factorization realizable iff major realizable}.
	\end{proof}
	
	The above theorem can be generalized to flag matroids. A \emph{flag matroid} of length $l$ is a chain of matroid quotients of the form $M_1 \overset{f_1}{\twoheadrightarrow} M_2 \overset{f_2}{\twoheadrightarrow} \dots \overset{f_{l}}{\twoheadrightarrow} M_{l+1}$, and a flag matroid is realizable over a field $K$ if it is so as a diagram of matroids. If we factor each quotient $f_i$ with the Higgs factorization, we obtain a chain of elementary quotients that constitute a factorization of the quotient $M_1 \twoheadrightarrow M_{l+1}$ (note that this is not necessarily equal to the Higgs factorization of this last quotient). We call such a factorization the \emph{Higgs factorization} of the flag matroid, and the major of this factorization the \emph{Higgs major} of the flag matroid.
	
	\begin{corollary}\label{cor: realizability of flags}
		Let $M_1 \overset{f_1}{\twoheadrightarrow} M_2 \overset{f_2}{\twoheadrightarrow} \dots \overset{f_{l}}{\twoheadrightarrow} M_{l+1}$ be a length $l$ flag matroid, and let $K$ be an infinite field. The following statements are equivalent:
		\begin{enumerate}
			\item the flag matroid is realizable over $K$,
			\item the Higgs factorization of the flag matroid is realizable over $K$,
			\item the Higgs major of the flag matroid is realizable over $K$.
		\end{enumerate}
	\end{corollary}
	\begin{proof}
		For $(1) \Rightarrow (2)$ it is enough to apply Proposition \ref{prop: map induces factorization} on each quotient $f_i$. The rest of the proof is analogous to the proof of Theorem \ref{thm: characterization of realizability of quotients}.
	\end{proof}
	
	Now we interpret Theorem \ref{thm: characterization of realizability of quotients} in terms of maps of realization spaces. Let $[n] = \{ 1,\dots,n \}$, and denote by $\binom{[n]}{r}$ the subsets of $[n]$ of cardinality $r$. Let $M$ be a rank $r$ matroid on $[n]$. The \emph{realization space} of $M$ is the subset
	\[ \mathcal{R}(M) = \left\{ (p_B)_{B \in \binom{[n]}{r}} \in \Gr(r,n) : p_B \neq 0 \Leftrightarrow B \in \mathcal{B}(M) \right\} \subseteq \Gr(r,n). \]
	If $N$ is a rank $s \leq r$ matroid on $[n]$ such that $M \twoheadrightarrow N$, then the \emph{realization space} of this quotient is the subset
	\[ \begin{split}
		\mathcal{R}(M \twoheadrightarrow N) = \Big\{ (p_B)_{B \in \binom{[n]}{r} \cup \binom{[n]}{s}} \in \Fl(s,r,n) : p_B \neq 0 \Leftrightarrow B \in \mathcal{B}(M) \cup \mathcal{B}(N) \Big\} \\
		\subseteq \Fl(s,r;n) \subseteq \Gr(s,n) \times \Gr(r,n).
	\end{split} \]
	Now let $H$ be the Higgs major of $M \twoheadrightarrow N$ on the ground set $[n] \cup [\overline{k}]$ where $[\overline{k}] = \{ \overline{1},\overline{2},\dots,\overline{k} \}$. Then, there is a surjective map $\varphi: \mathcal{R}(H) \twoheadrightarrow \mathcal{R}(M \twoheadrightarrow N)$ defined by
	\[ \varphi \left( (p_B)_{B \in \binom{[n] \cup [\overline{k}]}{r}} \right) = (p_{B'})_{B' \in \binom{[n]}{r}} \times (p_{B'' \cup [\overline{k}]})_{B'' \in \binom{[n]}{s}}. \]
	In other words, $\varphi$ is projecting away the coordinates $p_B$ of $\mathcal{R}(H)$ for all $B \in \binom{[n] \cup [\overline{k}]}{r}$ that are not contained in $[n]$ and do not contain $[\overline{k}]$. Since $\mathcal{R}(M \twoheadrightarrow N) \subseteq \mathcal{R}(M) \times \mathcal{R}(N)$, a point in the realization space of $M \twoheadrightarrow N$ can be viewed as a particular pair of realizations of $M$ and $N$ respectively. Consider a realization of $H$, corresponding to a point $(p_B)_{B \in \binom{[n] \cup [\overline{k}]}{r}}$ in $\mathcal{R}(H)$. The map $\varphi$ is sending this realization of $H$ to the pair of realizations of $M$ and $N$ by identifying the realization of $M$ with the point in $\mathcal{R}(M)$ with coordinates $(p_{B'})_{B' \in \binom{[n]}{r}}$ and the realization of $N$ with the point in $\mathcal{R}(N)$ with coordinates $(p_{B'' \cup [\overline{k}]})_{B'' \in \binom{[n]}{s}}$.
	
	The fact that $\varphi$ is well defined can be verified directly from comparing the Pl\"{u}cker relations of $\mathcal{R}(H)$ and the incidence Pl\"{u}cker relations of $\mathcal{R}(M \twoheadrightarrow N)$. The surjectivity of $\varphi$ is essentially a reformulation of $(1) \Rightarrow (3)$ of Theorem \ref{thm: characterization of realizability of quotients}.

	\section{Applications and examples}\label{sec: applications}
	
	In this section, we explain how quotients of matroids relate in a natural way to relative realizability problems in tropical geometry.
	
	\subsection{Linear case}\label{sec: linear case}
	
	Let $M$ be a rank $r$ matroid with ground set $E = \{0,1,\dots,n\}$. We recall that a \emph{cycle} of a matroid is a union of circuits. Denote by $\mathcal{C}(M)$ the set of circuits of $M$, and by $\mathcal{V}(M)$ the set of cycles of $M$. A family $\mathcal{V}$ of subsets of $E$ is the set of cycles of a matroid if and only if it satisfies the following axioms \cite{brylawski1986appendix}:
	\begin{enumerate}\label{axioms: cycles}
		\item $\emptyset \in \mathcal{V}$,
		\item $\mathcal{V}$ is closed under taking unions,
		\item if $V_1,V_1 \in \mathcal{V}$ and $e \in V_1 \cap V_2$ there exists $V \in \mathcal{V}$ such that
		\[ (V_1 \cup V_2) \setminus (V_1 \cap V_2) \subseteq V \subseteq (V_1 \cup V_2) \setminus \{ e \}. \]
	\end{enumerate}
	\noindent
	For a subset $A \subseteq E$ define $\trop(A)$ to be the set of vectors $v = (v_0,\dots,v_n) \in \mathbb{R}^{n+1}$ such that the minimum of the numbers $v_i$ is attained at least twice as $i$ ranges in $A$. Since if $v \in \trop(A)$ then also $v+\lambda \mathbf{1} \in \trop(A)$ for every $\lambda \in \mathbb{R}$, where $\mathbf{1} = (1,\dots,1) \in \mathbb{R}^{n+1}$, we regard $\trop(A)$ as a subset of the quotient space $\mathbb{R}^{n+1} / \mathbf{1}\mathbb{R} \simeq \mathbb{R}^n$.
	
	\begin{proposition definition}[Tropical linear space]\label{prop-def: tropical linear space}
		Let $M$ be a matroid on $E = \{ 0,1,\dots,n \}$. The \emph{tropical linear space} of $M$ is the set
		\[ \trop(M) = \bigcap_{C \in \mathcal{C}(M)} \trop(C) = \bigcap_{V \in \mathcal{V}(M)} \trop(V) \subseteq \mathbb{R}^n. \]
	\end{proposition definition}
	\begin{proof}
		We need to prove the second equality of the above equation. The inclusion $\bigcap_{C \in \mathcal{C}(M)} \trop(C) \supseteq \bigcap_{V \in \mathcal{V}(M)} \trop(V)$ is obvious since every circuit is a cycle by definition. For the reverse inclusion, let $v \in \bigcap_{C \in \mathcal{C}(M)} \trop(C)$ and let $V$ be a cycle of $M$. It is enough to show that $v \in \trop(V)$. Suppose that the minimum of the coordinates $v_i$ of $v$, for $i$ ranging in $V$, is attained at $v_j$ for some $j \in V$. Since $V$ is a union of circuits, there exists a circuit $C$ of $M$ that contains $j$. Now, by the choice of $v$, we have $v \in \trop(C)$, so the minimum of the numbers $v_i$ for $i$ ranging in $C$ is attained at least twice. Since $C \subseteq V$ and $v_j$ was the minimum of the numbers $v_i$ for $i \in V$, the minimum in $V$ is also attained at least twice, hence $v \in \trop(V)$.
	\end{proof}
	
	We can give the set $\trop(M)$ various fan structures, and such a fan is usually called \emph{Bergman fan} (see \cite[Section 4.2]{maclagan2021introduction} or \cite{feichtner2004chow} for more information).
	
	Let $M_1$ and $M_2$ be two matroids on the common ground set $\{ 0,1,\dots,n \}$. From \cite[Proposition 7.4.7]{white1986theory} we have that $M_1 \twoheadleftarrow M_2$ if and only if every circuit of $M_2$ is a union of circuits of $M_1$ (i.e. a cycle of $M_1$). It is easy to see that this is equivalent to the condition that every cycle of $M_2$ is a cycle of $M_1$. This observation, together with Proposition-Definition \ref{prop-def: tropical linear space} imply the following fact.
	
	\begin{corollary}\label{cor: inclusion of tropical linear spaces}
		Let $M_1$ and $M_2$ be two matroids on a common ground set. We have
		\[ M_1 \twoheadleftarrow M_2 \Longleftrightarrow \trop(M_1) \subseteq \trop(M_2). \]
	\end{corollary}
	
	Let $I$ be a homogeneous ideal of $K[x_0,\dots,x_n]$. We denote by $\trop(V(I))$ the tropicalization of $V(I) \cap T^n$, where $T^n$ is the algebraic torus of $\mathbb{P}^n$. If $K$ is infinite and $I$ is linear then the supports of the degree one polynomials in $I$ are the cycles of a matroid $M$, called the underlying matroid of $I$ (see Lemma \ref{lem: degree d matroid} below). In this situation we have the following (set-theoretic) equality: $\trop(V(I)) = \trop(M)$ (see \cite[Chapter 4]{maclagan2021introduction}).
	
	In this paper, all tropicalizations are done with respect to the trivial valuation. By a \emph{tropical variety} here we mean a pure-dimensional balanced polyhedral complex obtained as the tropicalization of an algebraic variety, that is, an object of the form $T = \trop(V(I))$ for some homogeneous ideal $I$ of $K[x_0,\dots,x_n]$.  Let $T_1$ and $T_2$ be two tropical varieties such that $T_1 \subseteq T_2$. We say that such an inclusion is \emph{realizable} over a field $K$ if there exists $I_1$ and $I_2$ two homogeneous ideals of $K[x_0,\dots,x_n]$ such that $T_i = \trop(V(I_i))$ for $i \in \{1,2\}$ and $V(I_1) \subseteq V(I_2)$.
	
	\begin{proposition}\label{prop: linear relative realizability}
		Let $T_i = \trop(V(I_i))$ for $i \in \{ 1,2 \}$ be two tropical varieties such that $T_1 \subseteq T_2$, where $I_1$ and $I_2$ are homogeneous linear ideals of $K[x_0,\dots,x_n]$, and let $M_1$ and $M_2$ be the underlying matroids of $I_1$ and $I_2$ respectively. The inclusion $T_1 \subseteq T_2$ is realizable over $K$ if and only if so is the quotient of matroids $M_1 \twoheadleftarrow M_2$.
	\end{proposition}
	\begin{proof}
		First assume that $T_1 \subseteq T_2$ is realizable over $K$. This means that there exist two homogeneous ideals $J_1$ and $J_2$ of $K[x_0,\dots,x_n]$ such that $T_i = \trop(V(J_i))$ for $i \in \{1,2\}$ and $V(J_1) \subseteq V(J_2)$. Since $T_1$ and $T_2$ are linear tropical varieties, this implies that $J_1$ and $J_2$ are linear ideals (see for instance \cite[Theorem 4.7]{fink2013tropical}). In particular, $J_1$ and $J_2$ have the same underlying matroid of $I_1$ and $I_2$ respectively. Now let $U_i$ be the vector space of polynomials of degree one in $J_i$, for $i \in \{1,2\}$. Since $V(J_1) \subseteq V(J_2)$ and $J_1$ and $J_2$ are linear, in particular radical, we have $J_1 \supseteq J_2$, so $U_1 \supseteq U_2$, and this last inclusion is a realization of $M_1 \twoheadleftarrow M_2$ in the sense we discussed in Section \ref{sec: realizability}.
		
		Conversely, assume that $M_1 \twoheadleftarrow M_2$ is realizable. Then there exist two linear subspaces $U_1 \supseteq U_2$ of $K^{n+1}$ that realize this quotient. The vectors in $U_i$ can be thought as linear polynomials in $K[x_0,\dots,x_n]$, and we let $J_i$ be the linear homogeneous ideal generated by these polynomials, for $i \in \{1,2\}$. By construction $J_1 \supseteq J_2$, so $V(J_1) \subseteq V(J_2)$, and $T_i = \trop(V(J_i))$ as $J_i$ and $I_i$ have the same underlying matroid, for $i \in \{1,2\}$.
	\end{proof}
	
	\subsection{Tropical Fano schemes}
	
	In \cite{lamboglia2022tropical} Lamboglia studied \emph{tropical Fano schemes} and tropicalizations of Fano schemes. The Fano scheme $F_r(X)$ of a projective variety $X$ is the fine moduli space parametrizing linear spaces of dimension $r$ contained in $X$. The tropicalization $\trop(F_r(X))$ of $F_r(X)$ inside $\trop G(r+1,n+1)$ parametrizes the tropicalization of linear spaces contained in $X$. On the other hand, the tropical Fano scheme $F_r(\trop(X))$ of the tropical variety $\trop(X)$ is a tropical prevariety parametrizing tropicalized linear spaces of dimension $r$ contained in $\trop(X)$. From the definition, it is immediate to observe that
	\[ \trop(F_r(X)) \subseteq F_r(\trop(X)). \]
	However, the above inclusion might be strict, already when $X$ is just a linear space. In fact, \cite[Theorem 3.1]{lamboglia2022tropical} describes a family of planes in $\mathbb{P}^n$ with this property. Every plane $L$ of such family tropicalizes to the standard tropical plane in $\mathbb{R}^n \simeq \mathbb{R}^{n+1} / \mathbf{1}\mathbb{R}$. The tropicalization thus contains the (realizable) tropical line $\Gamma$ whose rays are $\pos(e_i,e_{i+1})$ for $i$ ranging in all the even numbers in $\{ 0,1,\dots,n-1 \}$. However, no line contained in $L$ tropicalize to $\Gamma$.
	
	Now the combination of Proposition \ref{prop: linear relative realizability} and Theorem \ref{thm: characterization of realizability of quotients} offers another perspective on the above phenomenon. In fact, since $L$ tropicalize to the standard tropical plane, its underlying matroid is $U_{3,n+1}$. Further, the underlying matroid of $\Gamma$ is the matroid $M_2$ of rank $2$ on ground set $E = \{ 0,\dots,n \}$ with no loops and parallel elements given by the disjoint pairs of consecutive numbers: $\{ 0,1 \}$, $\{ 2,3 \}$ up to $\{ n-1,n \}$ or $\{ n-2,n-1 \}$ depending on the parity of $n$. Now since $\Gamma \subseteq \trop(L)$, from Corollary \ref{cor: inclusion of tropical linear spaces} we have $U_{3,n+1} \twoheadrightarrow M_2$. The Higgs major of this quotient is a matroid $M_3 = U_{3,n+1} + e$ of rank $3$ on the ground set $\{ 0,\dots,n \} \cup \{ e \}$ where the cardinality $3$ dependent subsets are all of the form $\{ i,i+1,e \}$ where $\{ i,i+1 \}$ are parallel elements in $M_2$. Such a matroid is realizable over $\mathbb{C}$, and from Theorem \ref{thm: characterization of realizability of quotients} every realization of $M_3$ gives us a realization of the quotient $U_{3,n+1} \twoheadrightarrow M_2$, which in turn, from Proposition \ref{prop: linear relative realizability}, provides us a pair of a line contained in a plane $\ell \subseteq L$ that tropicalize to $\Gamma$ and the standard tropical plane respectively. The crucial point here is that if we fix a realization of $U_{3,n+1}$ over $\mathbb{C}$, it is not guaranteed that this realization can be extended to a realization of $M_3$. This means that we cannot realize the quotient $U_{3,n+1} \twoheadrightarrow M_2$ with this fixed realization of $U_{3,n+1}$. We further illustrate this phenomenon with the examples discussed in \cite{lamboglia2022tropical}.
	
	\begin{example}[{\cite[Example 3.4]{lamboglia2022tropical}}]
		Let $L_1 \subseteq \mathbb{P}^5$ be the plane spanned by the rows of the following matrix
		\begin{equation*}
			\begin{bmatrix}
			1 & 3 & 0 & 1 & 5 & 7 \\
			0 & 0 & 1 & 3 & -1 & -1 \\
			1 & 4 & -1 & -3 & 0 & 0
			\end{bmatrix} \in \mathbb{C}^{3,6}.
		\end{equation*}
		The above matrix realizes $U_{3,6}$ over $\mathbb{C}$, therefore the tropicalization of $L_1$ is the standard tropical plane. Now this realization can be extended to a realization of $M_3$ given by the following matrix, obtained by adding the column vector $(1,0,0) \in \mathbb{C}^3$:
		\begin{equation*}
		\begin{bmatrix}
		1 & 3 & 0 & 1 & 5 & 7 & 1 \\
		0 & 0 & 1 & 3 & -1 & -1 & 0 \\
		1 & 4 & -1 & -3 & 0 & 0 & 0
		\end{bmatrix} \in \mathbb{C}^{3,7}.
		\end{equation*}
		Now, by simply quotienting by the vector $(1,0,0) \in \mathbb{C}^3$ corresponding to the element $e$ of $M_3 = U_{3,6} + e$, we obtain the following realization of $M_2$:
		\begin{equation*}
		\begin{bmatrix}
		0 & 0 & 1 & 3 & -1 & -1 \\
		1 & 4 & -1 & -3 & 0 & 0
		\end{bmatrix} \in \mathbb{C}^{2,6}.
		\end{equation*}
		The rowspace of the above matrix is a $2$-dimensional vector space in $\mathbb{C}^6$ corresponding to a line $\ell$ in $\mathbb{P}^5$ contained in $L_1$, that tropicalizes to $\Gamma$.
	\end{example}
	
	\begin{example}[{\cite[Example 3.3]{lamboglia2022tropical}}]
		Let $L_2 \subseteq \mathbb{P}^5$ be the plane spanned by the rows of the following matrix
		\begin{equation*}
			\begin{bmatrix}
			0 & -271 & -92 & 0 & -13 & -54 \\
			0 &	-18 & -7 & -1 & 0 & -4 \\
			-1 & 12293 & 4173 & 0 & 588 & 2450
			\end{bmatrix} \in \mathbb{C}^{3,6}.
		\end{equation*}
		The above matrix realizes $U_{3,6}$ over $\mathbb{C}$, therefore the tropicalization of $L_1$ is the standard tropical plane. Now this realization cannot be extended to a realization of $M_3$. In fact, a realization of $M_3$ would be given by a matrix of the following form
		\begin{equation*}
			A =
			\begin{bmatrix}
				0 & -271 & -92 & 0 & -13 & -54 & x_1\\
				0 &	-18 & -7 & -1 & 0 & -4 & x_2 \\
				-1 & 12293 & 4173 & 0 & 588 & 2450 & x_3
			\end{bmatrix}
			.
		\end{equation*}
		In particular, the determinant of the square submatrices indexed by the dependent subsets of cardinality $3$ should be zero. Therefore we have:
		\begin{equation*}
			\begin{array}{ccccc}
				\det( A_{0,1,e} ) = & -18 x_1 & + 271 x_2 &  &= 0\\
				\det( A_{2,3,e} ) = & 4173 x_1 & & + 92 x_3 &= 0\\
				\det( A_{4,5,e} ) = & 2352 x_1 & + 98 x_2 & + 52 x_3 &= 0
			\end{array}
		\end{equation*}
		The above linear system has only one solution: $(x_1,x_2,x_3) = (0,0,0)$, but this is impossible since $M_3$ has no loops. Therefore, the realization of $U_{3,6}$ given by $L_2$ cannot be extended to a realization of $M_3$. Thus, by Theorem \ref{thm: characterization of realizability of quotients} no realization of the quotient $U_{3,6} \twoheadrightarrow M_2$ can be constructed with such a realization of $U_{3,6}$. From Proposition \ref{prop: linear relative realizability} and Theorem \ref{thm: characterization of realizability of quotients}, this implies that there is no line $\ell$ contained in $L_2$ tropicalizing to $\Gamma$.
	\end{example}

	\subsection{Nonlinear case}\label{sec: nonlinear case}
	
	Let $I$ be a homogeneous ideal of $K[x_0,\dots,x_n]$. For every $d \in \mathbb{N}$ denote by $I_d$ the the degree $d$ part of $I$, that is, the set of polynomials of degree $d$ in $I$. Let $\Mon_d$ be the set of monomials of degree $d$ in $K[x_0,\dots,x_d]$. The \emph{support} $\supp(f)$ of a polynomial $f \in K[x_0,\dots,x_n]$ is the set of monomials with nonzero coefficients in $f$.
	
	\begin{lemma}\label{lem: degree d matroid}
		Let $K$ be an infinite field. For every $d \in \mathbb{N}$ the set of supports of polynomials in $I_d$ is the set of cycles of a matroid $M(I_d)$.
	\end{lemma}
	\begin{proof}
		We need to show that the set $\mathcal{V}$ of supports of $I_d$ satisfies the axioms \ref{axioms: cycles}. The zero vector is in $I_d$ by convention, so $\emptyset \in \mathcal{V}$ and $\mathcal{V}$ satisfies axiom $1$.
		
		For axiom $2$, let $V_1,V_2 \in \mathcal{V}$, and let $f_1,f_2 \in I_d$ such that $\supp(f_i) = V_i$ for $i \in \{ 1,2 \}$. Let $a_m$ (resp. $b_m$) be the coefficient in $f_1$ (resp. $f_2$) of the monomial $m \in \Mon_d$. Since $K$ is infinite, there exists $c \in K$ that is not equal to $-a_m / b_m$ for every $m \in \Mon_d$ such that $b_m \neq 0$. By construction, there are no cancellations in the polynomial $f = f_1 + c f_2 \in I_d$, therefore $V = \supp(f) = V_1 \cup V_2 \in \mathcal{V}$.
		
		For axiom $3$, let $V_1,V_2 \in \mathcal{S}$ and $e \in V_1 \cap V_2$, then there exists $f_1,f_2 \in I_d$ such that $V_i = \supp(f_i)$ for $i \in \{ 1,2 \}$. Now let $c_i \in K$ be the coefficient of the monomial $e$ in $f_i$  for $i \in \{ 1,2 \}$. Then the polynomial $f = f_1 - (c_1/c_2) f_2 \in I_d$ has support $V = \supp(f) \in \mathcal{V}$ such that $V \subseteq (V_1 \cup V_2) \setminus \{ e \}$ by construction, further $(V_1 \cup V_2) \setminus (V_1 \cap V_2) \subseteq V$ as no cancellation can happen outside $V_1 \cap V_2$.
	\end{proof}
	
	
	Let $\nu_{d}: \mathbb{P}^n \rightarrow \mathbb{P}^{N}$ be the degree $d$ Veronese embedding, where $N = \binom{n+d}{d}-1$. Let $U_{n,d} = \{ u \in \mathbb{N}^{n+1} :\sum_i u_i = d \}$ and let $K[z_u : u \in U_{n,d}]$ be the coordinate ring of $\mathbb{P}^N$. Now, let $I$ be a homogeneous ideal of $K[x_0,\dots,x_n]$ and choose $d \in \mathbb{N}$ such that $V(I) = V(I_d)$. We have
	\begin{equation}\label{eq: veronese linear}
	\nu_d(V(I)) = \nu_d(\mathbb{P}^n) \cap L_{I,d}
	\end{equation}
	where $L_{I,d}$ is some linear subspace of $\mathbb{P}^N$ which linear equations are obtained by substituting $x^u$ with $z_u$ for every $u \in U_{n,d}$ in the polynomials $f \in I_d$ (see \cite[Exercise 2.9]{harris2013algebraic}).	Further, note that, if $K$ is infinite, the underlying matroid of the linear space $L_{I,d}$ is $M(I_d)$, hence $\trop(L_{I,d}) = \trop(M(I_d))$.
	
	Let $A \in \mathbb{R}^{N,n+1}$ be the matrix whose rows consists of the vectors $u \in U_{n,d}$. Let $\ell_A: \mathbb{R}^{n+1} \rightarrow \mathbb{R}^N$ be the linear map given by the multiplication by $A$. Since $\ell_A(\mathbf{1}_{n+1}) = d \cdot \mathbf{1}_N$, the map $\ell_A$ induces a linear map $\trop(\nu_d): \mathbb{R}^n \rightarrow \mathbb{R}^N$ obtained by quotienting by $\mathbf{1}_{n+1}$ and identifying $\mathbb{R}^n$ with the quotient space $\mathbb{R}^{n+1} / \mathbf{1}\mathbb{R}$. Since the Veronese embedding $\nu_d$ is a monomial map, from \cite[Corollary 3.2.13]{maclagan2021introduction} we have the following result.
	
	\begin{lemma}\label{lem: veronese commutes with tropicalization}
		Let $I$ be an homogeneous ideal of $K[x_0,\dots,x_n]$. For every $d \in \mathbb{N}$ we have
		\[ \trop( \nu_{d}(V(I)) ) = \trop(\nu_d)(\trop(V(I)). \]
	\end{lemma}
	
	\begin{lemma}\label{lem: veronese trop distribution}
		Let $I$ be a homogeneous ideal of $K[x_0,\dots,x_n]$ and let $d \in \mathbb{N}$ such that $\trop(V(I)) = \bigcap_{f \in I_d} \trop(V(f))$ and $V(I) = V(I_d)$. We have
		\[ \trop(\nu_d(V(I)) = \trop(\nu_d(\mathbb{P}^n)) \cap \trop(L_{I,d}). \]
	\end{lemma}
	\begin{proof}
		From \eqref{eq: veronese linear} we have $\trop(\nu_d(V(I)) \subseteq \trop(\nu_d(\mathbb{P}^n)) \cap \trop(L_{I,d})$. For the reverse inclusion, let $v \in \trop(\nu_d(\mathbb{P}^n)) \cap \trop(L_{I,d})$. From Lemma \ref{lem: veronese commutes with tropicalization} we have that $v \in \trop(\nu_d(\mathbb{P}^n)) = \trop(\nu_d)(\mathbb{R}^n)$, that is $v$ is in the image of $\trop(\nu_d)$, so there exists $w \in \mathbb{R}^{n+1} / \mathbf{1} \mathbb{R} \simeq \mathbb{R}^n$ such that $v = \trop(\nu_d)(w)$. Now $v = \trop(\nu_d)(w) \in \trop(L_{I,d})$ implies that the minimum in $\trop(f)(w)$ is achieved at least twice for every $f \in I_d$, that is $w \in \bigcap_{f \in I_d} \trop(V(f)) = \trop(V(I))$. Hence $v = \trop(\nu_d)(w) \in \trop(\nu_d)(\trop(V(I))) = \trop(\nu_d(V(I)))$, where in the last equality we applied Lemma \ref{lem: veronese commutes with tropicalization}.
	\end{proof}

	\begin{proposition}\label{prop: quotient implies inclusion}
		Let $K$ be an infinite field, and let $I$ and $J$ be two homogeneous ideals of $K[x_0,\dots,x_n]$. There exists $d_0 \in \mathbb{N}$ such that for every $d \geq d_0$ we have
		\[ M(I_d) \twoheadleftarrow M(J_d) \Longrightarrow \trop(V(I)) \subseteq \trop(V(J)). \]
	\end{proposition}
	\begin{proof}
		By applying \cite[Theorem 3.7]{alessandrini2013tropicalization} we can choose $d_0 \in \mathbb{N}$ sufficiently large such that for every $d \geq d_0$ the hypothesis of Lemma \ref{lem: veronese trop distribution} are satisfied for both $I$ and $J$. Since the underlying matroid of $L_{I,d}$ (resp. $L_{J,d}$) is $M(I_d)$ (resp. $M(J_d)$), from Proposition-Definition \ref{prop-def: tropical linear space} we have $\trop(L_{I,d}) \subseteq \trop(L_{J,d})$. By intersecting with $\trop(\nu_d(\mathbb{P}^n))$ and applying Lemma \ref{lem: veronese commutes with tropicalization} and Lemma \ref{lem: veronese trop distribution} we obtain
		\[ \trop(\nu_d)(\trop(V(I))) \subseteq \trop(\nu_d)(\trop(V(J))) \]
		which implies $\trop(V(I)) \subseteq \trop(V(J))$ since $\trop(\nu_d)$ is injective.
	\end{proof}
	
	The reverse implication of the above proposition is not true in general, as the following example shows.
	
	\begin{example}\label{ex: standard tropical line}
		Let $I = (x_0+x_1+x_2)$ and $J = (x_0^p+x_1^p+x_2^p)$ be two homogeneous ideals in $K[x_0,x_1,x_2]$, for some prime $p$. The tropicalization of the varieties of these two ideals is (set-theoretically) the same: the standard tropical line in $\mathbb{R}^2$. In this situation, the number $d_0$ of Proposition \ref{prop: quotient implies inclusion} can be chosen equal to $p$, and we have that, if $K$ has characteristic not equal to $p$, $M(I_p)$ is not a quotient of $M(J_p)$. 
	\end{example}
	
	\begin{remark}\label{rmk: symmetric powers}
		Another situation in which the reverse implication of Proposition \ref{prop: quotient implies inclusion} does not hold concerns \emph{symmetric powers} of matroids. In fact, as explained in \cite[Remark 2.18]{anderson2023matroid}, if a matrix $A$, with coefficients in some field $K$, realizes a matroid $M$, the $d$-th Macaulay matrix of $A$ realizes a $d$-th symmetric power of $M$. However, such a $d$-th symmetric power depends both on the characteristic of $K$ and the matrix $A$ we started with (not just on the matroid $M$). Thus, from two distinct realizations of $M$, we obtain two distinct realizable symmetric powers of $M$. These two matroids will then be the degree $d$ part of two ideals tropicalizing to the same tropical variety, despite the fact that the two underlying matroids in degree $d$ are distinct (so one is not the quotient of the other).
	\end{remark}
	
	
	\begin{remark}
		While Proposition \ref{prop: linear relative realizability} essentially reduces the relative realizability problem for tropical varieties of degree one to a problem of realizability of matroid quotients, the same cannot be done directly for tropical varieties of higher degrees. Example \ref{ex: standard tropical line} is already an example of this. One could still try to recover some non-realizability certificate from the non realizability of a finite list of matroid quotients, although an algorithmic implementation of such a procedure might not be feasible in practice.
	\end{remark}

	\section{Questions and future work}\label{sec: questions}
	
	Theorem \ref{thm: characterization of realizability of quotients} and Corollary \ref{cor: realizability of flags} characterize realizability of a (single) quotient of matroids and a flag of matroids in terms of realizability of a single matroid. It is natural to ask if something similar can be done more generally for a diagram of matroids.
	
	\begin{question}
		Can we characterize the realizability of a diagram of matroids in terms of the realizability of a single matroid?
	\end{question}
	
	Here we restricted ourselves to (classical) matroids. In light of the connection with relative realizability problems in tropical geometry, it is natural to ask if a similar theory could be developed for valuated matroids. Quotients of valuated matroids were studied in \cite{brandt2021tropical} and \cite{borzi2023linear}. In particular, one could try to define a \emph{valuated} Higgs lift, from which the notions of Higgs factorization and major could be constructed.
	
	\begin{question}
		Is there an analogue notion of Higgs lift for quotients of valuated matroids?
	\end{question}

	Let $\Mon_d$ be the set of monomials of degree $d$ in $K[x_0,\dots,x_n]$. Let $M_1$ and $M_2$ be two matroids on the ground set $\Mon_d$ for some $d >0$. Denote by
	\[ T_i = \trop(\nu_d)^{-1} \big( \trop(M_i) \cap \trop(\nu_d(\mathbb{P}^n) \big) \]
	for $i \in \{1,2\}$. In light of Remark \ref{rmk: symmetric powers} it is natural to ask the following question.
	
	\begin{question}
		Under which conditions on $M_1$ and $M_2$ we have $T_1 = T_2$?
	\end{question}

	\bibliographystyle{plain}
	\bibliography{Reference.bib}

\begin{thebibliography}{10}

\bibitem{alessandrini2013tropicalization}
Daniele Alessandrini and Michele Nesci.
\newblock On the tropicalization of the {H}ilbert scheme.
\newblock {\em Collectanea mathematica}, 64:39--59, 2013.

\bibitem{anderson2023matroid}
Nicholas Anderson.
\newblock Matroid products in tropical geometry.
\newblock {\em arXiv preprint arXiv:2306.14771}, 2023.

\bibitem{bertrand2019planar}
Beno{\^\i}t Bertrand, Erwan Brugall{\'e}, and Luc{\'\i}a~L{\'o}pez De~Medrano.
\newblock Planar tropical cubic curves of any genus, and higher dimensional
  generalisations.
\newblock {\em L’Enseignement Math{\'e}matique}, 64(3):415--457, 2019.

\bibitem{borovik2003coxeter}
Alexandre~V Borovik, Izrail~Moiseevich Gelfand, and Neil White.
\newblock {\em Coxeter matroids}.
\newblock Springer, 2003.

\bibitem{borzi2023linear}
Alessio Borz{\`\i} and Victoria Schleis.
\newblock Linear degenerate tropical flag varieties.
\newblock {\em arXiv preprint arXiv:2308.04193}, 2023.

\bibitem{bossinger2017computing}
Lara Bossinger, Sara Lamboglia, Kalina Mincheva, and Fatemeh Mohammadi.
\newblock Computing toric degenerations of flag varieties.
\newblock {\em Combinatorial Algebraic Geometry: Selected Papers From the 2016
  Apprenticeship Program}, pages 247--281, 2017.

\bibitem{brandt2021tropical}
Madeline Brandt, Christopher Eur, and Leon Zhang.
\newblock Tropical flag varieties.
\newblock {\em Advances in Mathematics}, 384:107695, 2021.

\bibitem{brylawski1986appendix}
Thomas Brylawski.
\newblock Appendix of {M}atroid {C}ryptomorphisms.
\newblock {\em Theory of matroids}, (26):298, 1986.

\bibitem{cameron2017flag}
Amanda Cameron, Rodica Dinu, Mateusz Micha{\l}ek, and Tim Seynnaeve.
\newblock Flag matroids: algebra and geometry.
\newblock In {\em International Conference on Interactions with Lattice
  Polytopes}, pages 73--114. Springer, 2017.

\bibitem{eur2020logarithmic}
Christopher Eur and June Huh.
\newblock Logarithmic concavity for morphisms of matroids.
\newblock {\em Advances in Mathematics}, 367:107094, 2020.

\bibitem{feichtner2004chow}
Eva~Maria Feichtner and Sergey Yuzvinsky.
\newblock Chow rings of toric varieties defined by atomic lattices.
\newblock {\em Inventiones mathematicae}, 155(3):515--536, 2004.

\bibitem{fink2013tropical}
Alex Fink.
\newblock Tropical cycles and {C}how polytopes.
\newblock {\em Beitr{\"a}ge zur Algebra und Geometrie/Contributions to Algebra
  and Geometry}, 54:13--40, 2013.

\bibitem{geiger2020realizability}
Alheydis Geiger.
\newblock On realizability of lines on tropical cubic surfaces and the
  {B}rundu-{L}ogar normal form.
\newblock {\em Le Matematiche}, 75(2):651--671, 2020.

\bibitem{gelfand1987combinatorial}
Israel~M Gelfand, R~Mark Goresky, Robert~D MacPherson, and Vera~V Serganova.
\newblock Combinatorial geometries, convex polyhedra, and {S}chubert cells.
\newblock {\em Advances in Mathematics}, 63(3):301--316, 1987.

\bibitem{haque2012tropical}
Mohammad~Moinul Haque.
\newblock Tropical incidence relations, polytopes, and concordant matroids.
\newblock {\em arXiv preprint arXiv:1211.2841}, 2012.

\bibitem{harris2013algebraic}
Joe Harris.
\newblock {\em Algebraic geometry: a first course}, volume 133.
\newblock Springer Science \& Business Media, 2013.

\bibitem{heunen2018category}
Chris Heunen and Vaia Patta.
\newblock The category of matroids.
\newblock {\em Applied Categorical Structures}, 26:205--237, 2018.

\bibitem{higgs1968strong}
DA~Higgs.
\newblock Strong maps of geometries.
\newblock {\em Journal of Combinatorial Theory}, 5(2):185--191, 1968.

\bibitem{jell2022moduli}
Philipp Jell, Hannah Markwig, Felipe Rinc{\'o}n, and Benjamin Schr{\"o}ter.
\newblock Moduli spaces of codimension-one subspaces in a linear variety and
  their tropicalization.
\newblock {\em The Electronic Journal of Combinatorics}, pages P2--31, 2022.

\bibitem{kennedy1975majors}
Daniel Kennedy.
\newblock Majors of geometric strong maps.
\newblock {\em Discrete Mathematics}, 12(4):309--340, 1975.

\bibitem{Kung1986bookchapter}
Joseph P.~S. Kung.
\newblock Strong maps.
\newblock In {\em Theory of matroids}, volume~26 of {\em Encyclopedia Math.
  Appl.}, pages 224--253. Cambridge Univ. Press, Cambridge, 1986.

\bibitem{lamboglia2022tropical}
Sara Lamboglia.
\newblock Tropical {F}ano schemes.
\newblock {\em Bulletin of the London Mathematical Society}, 54(4):1249--1264,
  2022.

\bibitem{maclagan2021introduction}
Diane Maclagan and Bernd Sturmfels.
\newblock {\em Introduction to tropical geometry}, volume 161.
\newblock American Mathematical Society, 2021.

\bibitem{oxley2011}
James Oxley.
\newblock {\em {Matroid Theory}}.
\newblock Oxford University Press, 02 2011.

\bibitem{panizzut2022tropical}
Marta Panizzut and Magnus~Dehli Vigeland.
\newblock Tropical lines on cubic surfaces.
\newblock {\em SIAM Journal on Discrete Mathematics}, 36(1):383--410, 2022.

\bibitem{schofield1992general}
Aidan Schofield.
\newblock General representations of quivers.
\newblock {\em Proceedings of the London Mathematical Society}, 3(1):46--64,
  1992.

\bibitem{white1986theory}
Neil White.
\newblock {\em Theory of matroids}.
\newblock Number~26. Cambridge University Press, 1986.

\end{thebibliography}
	
	\vspace{0.25cm}
	\noindent
	\textsc{Alessio Borz\`i, MPI MiS Leipzig, Inselstraße 22, 04103 Leipzig, Germany}\\
	\textit{Email:} alessio.borzi@mis.mpg.de 
	\vspace{0.25cm}
	
\end{document}